\documentclass[12pt,oneside]{amsart}
\usepackage{geometry}                
\usepackage{graphicx}
\usepackage{amssymb}
\usepackage{epstopdf}

\usepackage{amsmath,amsthm,amscd,amssymb, enumerate}
\usepackage{latexsym}
\usepackage[colorlinks,citecolor=red,pagebackref,hypertexnames=false]{hyperref}
\numberwithin{equation}{section}
\theoremstyle{plain}
\newtheorem{theorem}{Theorem}[section]
\newtheorem{lemma}[theorem]{Lemma}

\theoremstyle{definition}
\newtheorem{definition}[theorem]{Definition}

\theoremstyle{remark}
\newtheorem{remark}[theorem]{Remark}

\newtheorem{case[theorem]}{Case}

\def\R{\mathbb R}
\def\hd{{\dim_{\mathcal H}}}
\def\vt{{\bf\vec{t}}}

\date{September 4, 2014. First author supported by DMS-0853892 and -1362271; third author supported by an NSERC Discovery grant.}    
\address{Department of Mathematics, University of Rochester, Rochester, NY 14627}
\email{allan@math.rochester.edu}  
\address{Department of Mathematics, University of Rochester, Rochester, NY 14627}
\email{iosevich@math.rochester.edu}
\address{Department of Mathematics, University of British Columbia, Vancouver, B.C.
Canada V6T 1Z2}
\email{malabika@math.ubc.ca}

\title{On necklaces inside thin subsets of ${\Bbb R}^d$}

\author{Allan Greenleaf, Alex Iosevich and Malabika Pramanik} 

\begin{document}
\maketitle
\begin{abstract} We study similarity classes of point configurations
  in $\R^d$. Given a finite collection of points, a well-known
  question is:  How high does the Hausdorff dimension $\hd(E)$ of a
  compact set $E \subset {\Bbb R}^d$, $d \ge 2$, need to be to ensure
  that  $E$ contains some similar copy of this configuration? We prove
  results for a related problem, showing that for $\hd(D)$
  sufficiently large, $E$ must contain many point configurations that
  we call $k$-necklaces of constant gap, generalizing equilateral triangles and rhombuses in higher dimensions. Our results extend and complement those in \cite{CLP14,BIT14}, where  related questions were recently studied. 
\end{abstract}  

\maketitle
\tableofcontents

\section{Introduction} 

\vskip.125in 

The study of finite point configurations in sets of various sizes spans analysis, ergodic theory, number theory
and combinatorics. A corollary (due to Steinhaus) of the Lebesgue
density theorem states that any  measurable
set in $\mathbb R^d$ with positive Lebesgue measure contains a
similar copy of any finite configuration of points. There are
many variations on this result. For instance, instead of
sets of positive Lebesgue measure, one can consider an  unbounded set $E \subset {\Bbb R}^d$
of positive upper Lebesgue density, in the sense that 
\[ \limsup_{R \to \infty} \frac{\bigl|E \cap {[-R,R]}^d
  \bigr|}{{(2R)}^d}>0. \]
Here $| \cdot |$ denotes the $d$-dimensional Lebesgue measure. A
result of Bourgain \cite{B86} (also Furstenberg, Katznelson and Weiss \cite{FKW90}) proves that $E$ contains all sufficiently
large copies of a non-degenerate $k$-simplex, i.e., a $(k+1)$-point configuration, for $k \leq d$. Ziegler
\cite{Z06} has generalized this result for $k \geq d$, but the
sufficiently large copies of the configuration are shown to be contained
in an arbitrarily small neighborhood of $E$ rather than in $E$
itself.      
In particular, results of this type show that we can recover every
simplex similarity type inside a subset of ${\Bbb R}^d$ that is
``large'', either in the sense of positive Lebesgue measure or of
positive upper Lebesgue density. It is reasonable to wonder whether
similar conclusions continue to hold even if
such largeness assumptions are weakened. However, the following result
due to Maga \cite{Mag10} shows that the conclusion  in general fails
for Lebesgue null sets in $\mathbb R^d$, {\it even if the set under
consideration is of full Hausdorff dimension}.  Let $\hd(E)$ denote the Hausdorff dimension of a set $E\subset\R^d$.

\begin{theorem} \label{mag10} (Maga \cite{Mag10}) The following
  conclusions hold. 
\begin{enumerate}[(a)]
\item For any $d \ge 2$, there exists a  compact set $A \subset {\Bbb
    R}^d$ with $\hd(A)=d$ such that $A $ does not contain the vertices
  of any parallelogram. 
\item If $d=2$, then given any nondegenerate triple of points
  $x^1,x^2,x^3$ in $ \R^2$, there exists a compact set $A \subset
  {\Bbb R}^2$ with $\hd(A)=2$ such that $A$ does not contain the
  vertices of any triangle similar to $\bigtriangleup x^1x^2x^3$. 
\end{enumerate} 
\end{theorem} 

In view of Maga's result, it is reasonable to ask whether interesting specific 
point configurations can be found inside thin sets under additional
structural hypotheses. This question has been recently addressed by Chan,
\L aba and Pramanik  \cite{CLP14}, where the authors establish the
existence of certain finite point configurations in sets of sufficiently high Hausdorff dimension and  carrying a Borel measure
with decaying Fourier transform. (The measure should also satisfy
certain size bounds
for Euclidean balls.) The point configurations obtained in \cite{CLP14}
were required to obey appropriate nondegeneracy constraints when
expressed as a linear system, and included both geometric and
algebraic patterns such as corners in the plane, as well as  polynomial-type
configurations in $\mathbb R^d$.  However, some natural configurations do not satisfy the non-degeneracy
assumption of \cite{CLP14}. For example, {\it corners} in $\R^3$, defined as  collections of $4$ points $x,y,z,w$ in ${\Bbb R}^3$ such that 
\begin{equation} \label{3dcorner} \begin{aligned} (x-y) \perp (x-z),
    \quad &(x-y) \perp (x-z), \quad (x-z) \perp (x-w), \\
    &|x-y|=|x-y|=|x-w| \end{aligned} \end{equation} 
are not covered by the setup of \cite{CLP14}. Neither does a {\it nonplanar} (i.e., not necessarily planar) {\it rhombus} in $\R^3$, defined as a set of $4$ points $x,y,z,w$  such that $|x-y|=|y-z|=|z-w|=|w-x|$. 

It is reasonable to ask which point configurations can be recovered
without extra assumption on the Fourier decay. In view of Maga's
result  (Thm.~\ref{mag10} above), one cannot hope to prove
nontrivial results of this type for configurations that contain a
planar loop. However it still seems plausible that we may be able to
handle tree-like point configurations and loops that are not contained
in a plane and hence enjoy greater directional freedom. This question is partially addressed in \cite{BIT14}. To present this result, we need the following definition. 
\begin{definition} A {\it $k$-chain} in $E \subset {\Bbb R}^d$ {\it with gaps} ${\{t_i\}}_{i=1}^k$ is a sequence 
$$\{x^1,x^2, \dots, x^{k+1}: x^j \in E,\, \ |x^{i+1}-x^i|=t_i>0,\, \ 1 \leq i \leq k\}.$$ 
The $k$-chain has {\it constant gap} $t>0$ if all the $t_i=t$.
Finally, we say  that the chain is {\it non-degenerate} if all the $x^i$s are distinct.  \end{definition} See Fig. 1  for a depiction of  a 3-chain.
\begin{figure}
\label{chainfigure}
\centering
\includegraphics[scale=.5]{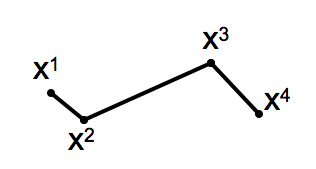}
\caption{A 3-chain}
\end{figure}
\begin{theorem} (Bennett, Iosevich and Taylor
  \cite{BIT14}) \label{mainbit} Suppose that   $E \subset {\Bbb R}^d$
  is a compact set,  $d \ge 2$, and that  $\hd(E)>\frac{d+1}{2}$. Then for any $k \ge 1$, there exists an open interval $I\subset\R$, such that for each $t \in I$ there exists a non-degenerate $k$-chain in $E$ with constant gap $t$. 
\end{theorem} 

The idea behind the proof of Thm.~\ref{mainbit} is to construct a
 measure on all  $k$-chains, naturally induced from a Frostman measure $\mu$  on $E$, and consider its Radon-Nikodym derivative. We prove that it is bounded from above in all cases, and from below
in the case when all the gaps are in a suitable interval. The lower
bound is accomplished using the continuity of the distance measure in
appropriate dimensional regimes. An upper bound is proved using a
fractal variant of the classical Parseval identity recently
established by Iosevich, Sawyer, Taylor and Uriarte-Tuero
\cite{ISTU14}, based on an earlier result of Strichartz
\cite{Str90}. In practice, this amounts to obtaining upper and lower bounds on the quantity 
\begin{equation} \label{chainradonnikodym} C_k^{\epsilon}(\mu)=\int
  \dots \int \prod_{j=1}^k \sigma^{\epsilon}_t(x^{j+1}-x^j)
  d\mu(x^j) \end{equation} that are uniform in $\epsilon$. Here and throughout the paper, $\sigma_t$ is the Lebesgue measure on the sphere of radius $t$, $\sigma_t^{\epsilon}=\sigma_t*\rho_{\epsilon}$, with $\rho \ge 0$ a smooth cut-off function, $\int \rho(x)=1$ and $\rho_{\epsilon}(x)=\epsilon^{-d} \rho \left(\frac{x}{\epsilon}\right)$. 
An analogous multilinear form, expressed in terms of the Fourier
transforms of measures rather than the measures themselves, was used
in \cite{CLP14} as well. There, a finite upper bound on the form
justified its existence and definition; a nontrivial lower  bound then
established the existence of the linear configurations.  

While the results in \cite{{BIT14},{CLP14}} are focused on point configurations that do not contain loops, we shall see that both the lower bound and the upper bound idea in \cite{BIT14}, combined with the generalized three-lines lemma approach in \cite{ISTU14}, allow us to capture configurations that were  inaccessible by these previous  methods. In particular, we will be able to handle nonplanar rhombuses in dimensions three and higher, as well as more complicated closed loops. We now turn our attention to the precise formulation of our results.

\section{Statement of Results} 

\begin{definition} A {\it $k$-necklace} in $E \subset {\Bbb R}^d$, $d \ge 2$,  with gaps $\vt=(t_1, t_2, \dots, t_{k})$, $t_j>0$, is a finite sequence $x^1, x^2, \dots, x^{k}$, $x^j \in E$, such that $|x^j-x^{j+1}|=t_j$, $1 \leq j \leq k-1$ and $|x^{k}-x^1|=t_{k}$. We say that this necklace is {\it non-degenerate} if $x^i \not=x^j$ for any $1 \leq j \leq k$, and {\it has constant gap $t$} if $t_1=\dots t_k=t$. \end{definition}


      \begin{figure}
\label{necklace}
\centering
\includegraphics[scale=.5]{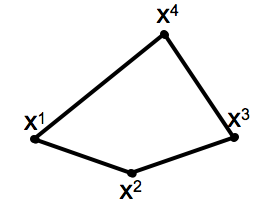}
\caption{A 4-necklace}
\end{figure}

\begin{remark} Thus, a $k$-necklace is a closed $(k+1)$-chain (see
  Fig. 2.), and being of constant gap is the same as all edges being
  of equal length. 
\end{remark}

\begin{remark} A nondegenerate 4-necklace of constant gap $t>0$  in
  $\R^d$  is  a nonplanar rhombus of side length
  $t$. 
  Note  that  two non-degenerate
  $4$-necklaces, even with similar gap vectors,  need not be similar to each other, due to the freedom that comes from not necessarily being planar. 
  \end{remark}

  \begin{remark}
  In general a $k$-necklace with a given gap vector $\vt$  is a member of the union of the similarity classes of a family of $k$-simplices, rather
  than being a similar copy of a specific $k$-simplex.  
  \end{remark}

  We now can state our main result.

\vskip.125in 

\begin{theorem} \label{necklace} Let $E$ be a compact subset of ${\Bbb R}^d$, $d \ge 3$. 

\vskip.125in 

i) Suppose that $d \ge 4$, $k$ is even and $dim_{{\mathcal H}}(E)>\frac{d+3}{2}$, without any additional assumptions on measures carried by $E$. Then there exists a non-empty open  interval $I$ such that for every $t \in I$, $E$ contains  some $k$-necklace with constant gap $ t$. 

\vskip.125in 

ii) Suppose that $d \ge 3$. Suppose that for some $\delta>0$,   $\hd(E)>d-\delta$ and there exists a Borel measure $\mu$ supported on $E$ such that

\begin{equation} \label{decay} |\widehat{\mu}(\xi)| \leq C{|\xi|}^{-1-\frac{\delta}{2}},\quad\forall\xi\in\R^d. \end{equation} 
Then  there exists a non-empty open interval $I$ such that for every $t \in I$, 
$E$ contains  a nonplanar rhombus of side length $t$. 


\end{theorem}


\begin{remark} It would be interesting to extend Thm. \ref{necklace} to cover the case when $k$ is odd. Note however that, at least in the case $k=3$, the conclusion of part (i) of  the theorem is certainly false in view of Maga's counter-example \cite{Mag10}. 
\end{remark} 


\begin{remark} If the $\hd(E)=s$, then in (\ref{decay}), $1+\frac{\delta}{2} \leq \frac{s}{2}$. In particular, if $E$ is a Salem set \cite{M95} of dimension $s>\frac{d+2}{2}$, then $E$ contains the vertices of a rhombus. \end{remark} 

\begin{remark} While we state Thm. \ref{necklace} for necklaces with constant gaps, a careful examination of the proof shows that we can say a bit more:
\end{remark}

\begin{definition} We say that a non-degenerate $(n-1)$-chain with vertices $x^1, x^2, \dots, x^n$ {\it generates }a non-degenerate $(2n-2)$-necklace with  vertices $x^1, x^2, \dots, x^{2n-2}$ if \\ $|x^j-x^{j+1}|=|x^{k+2-j}-x^{k+1-j}|$ for $2 \leq j \leq n-1$. (See Fig. 3.)  \end{definition} 

\noindent The proof of Thm. \ref{necklace} (i) shows that in fact  the conclusion holds for  any necklace with an even number of vertices which is generated by a non-degenerate chain.


\begin{figure}
\label{generate}
\centering
\includegraphics[scale=.5]{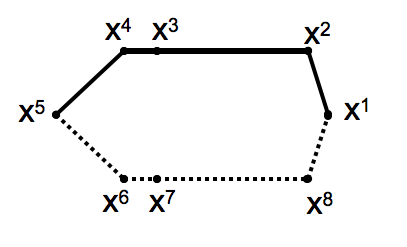}
\caption{A $4$-chain $x^1,\dots,x^5$ generates an $8$-necklace $x^1,\dots, x^8$.}
\end{figure}

\vskip.125in

\section{Proof of Theorem \ref{necklace}} 


\subsection{Preliminary calculations} 


We shall need the following result from \cite{ISTU14}, which we state in the form needed in this paper. 

\begin{theorem} (\cite[~Thm 1.1]{ISTU14}) \label{kick} Let $K\in\mathcal S'(\R^d)$ be a tempered distribution satisfying 
$$|\widehat{K}(\xi)| \leq C{|\xi|}^{-\gamma}, \ \gamma \in \left(0, \frac{d}{2} \right).$$ 
For $\epsilon>0$, let $K^{\epsilon}=K*\rho_{\epsilon}$.  
Suppose that $\phi, \psi$ are compactly supported Borel measures on ${\Bbb R}^d$ satisfying $\phi(B(x,r)) \leq Cr^{s_{\phi}}, \psi(B(x,r)) \leq Cr^{s_{\psi}}$, respectively, with $s_{\phi}, s_{\psi}>0$. Let $T_{K^{\epsilon}}f=K^{\epsilon}*(f\phi)$. Suppose that $\gamma>d-s$, where $s=\frac{s_{\phi}+s_{\psi}}{2}$. Then 
$$ {||T_{K^{\epsilon}}f||}_{L^2(\psi)} \leq C{||f||}_{L^2(\phi)}$$ where $C$ does not depend on $\epsilon$.
\end{theorem}  
\begin{proof} 
Since the proof of Theorem \ref{kick} is very simple, we include it for the sake of completeness. It is enough to show that 
$$ \langle T_{K^{\epsilon}}f, g\psi \rangle  \leq C{||f||}_{L^2(\phi)} \cdot {||g||}_{L^2(\psi)},\quad\forall f, g.$$ 
The left hand side equals 
$$ \int \widehat{K^{\epsilon}}(\xi) \widehat{f\phi}(\xi) \widehat{g\psi}(\xi) d\xi.$$ By the assumptions of Theorem \ref{kick}, the modulus of this quantity is bounded by 
$$ C \int {|\xi|}^{-\gamma} |\widehat{f\phi}(\xi)|
|\widehat{g\psi}(\xi)| d\xi,$$ and applying the Cauchy-Schwarz
inequality results in the following upper bound for this quantity: 
\begin{equation} \label{2square} C {\left( \int {|\widehat{f\phi}(\xi)|}^2 {|\xi|}^{-\gamma_{\phi}} d\xi \right)}^{\frac{1}{2}} \cdot 
{\left( \int {|\widehat{g \psi}(\xi)|}^2 {|\xi|}^{-\gamma_{\psi}} d\xi \right)}^{\frac{1}{2}} \end{equation} for any $\gamma_{\phi}, \gamma_{\psi}>0$ such that $\gamma=\frac{\gamma_{\phi}+\gamma_{\psi}}{2}$. 
By Lemma \ref{fenergy} soon to be proved below, the quantity (\ref{2square}) is bounded by $C {||f||}_{L^2(\phi)} \cdot {||g||}_{L^2(\psi)}$ after choosing, as we may, $\gamma_{\phi}>d-s_{\phi}$ and $\gamma_{\psi}>d-s_{\psi}$. This completes the proof of Theorem \ref{kick}. 
\end{proof} 

Let $\mu$ be a Frostman measure supported on $E$. Recursively define 
\begin{equation} \label{chainmeasures} d\mu_0(x):=d\mu(x); \ d\mu^{\epsilon}_{k+1}(x):=\sigma_t^{\epsilon}*\mu^{\epsilon}_k(x) d\mu(x) =: f_{k+1}(x) d\mu(x),\, k\ge 0. 
\end{equation} 

\begin{lemma} \label{chainenergylemma} Let $E\subset\R^d$ be compact with $\hd(E)>\frac{d+1}{2}$, and suppose that $\mu$ is a Frostman measure on  $E$. If $d-dim_{{\mathcal H}}(E)<\alpha<d$, then, with the notation in (\ref{chainmeasures}), 

\begin{equation} \label{chainenergy} \int {|\widehat{\mu^{\epsilon}_k}(\xi)|}^2 {|\xi|}^{-\alpha} d\xi \leq C(k)<\infty, \end{equation} where $C(k)$ is independent of $\epsilon$. \end{lemma} 


\begin{remark} A careful examination of the proof shows that $C(k)$ above depends on the $(d-\alpha)$-energy of $\mu$, namely $\int {|\widehat{\mu}(\xi)|}^2 {|\xi|}^{-\alpha} d\xi$. \end{remark} 

\begin{proof} 
This lemma is proved in \cite{BIT14}, but we give a proof for the sake of completeness. Begin by using (\ref{chainmeasures}) to rewrite the left hand side of (\ref{chainenergy}) in the form 
\begin{equation} \label{fkenergy} \int {|\widehat{f_k \mu}(\xi)|}^2 {|\xi|}^{-\alpha} d\xi. \end{equation} 


Assuming Lemma \ref{fenergy} as stated below for the moment and applying it to (\ref{fkenergy}), we see that 
$$ \int {|\widehat{f_k \mu}(\xi)|}^2 {|\xi|}^{-\alpha} d\xi \leq C{||f_k||}^2_{L^2(\mu)}.$$ We have thus reduced the issue to proving that ${||f_k||}_{L^2(\mu)}$ is bounded. 

Define the operator $Tf(x)=\sigma_t^{\epsilon}*(f\mu)$. Observe that $f_k(x)=Tf_{k-1}(x)$. By Theorem \ref{kick}, 
$$ {||f_k||}^2_{L^2(\mu)}={||Tf_{k-1}||}^2_{L^2(\mu)} \leq C {||f_{k-1}||}^2_{L^2(\mu)} \leq C^k {||f_1||}^2_{L^2(\mu)}$$
\begin{equation} \label{hinge} \leq C^k \int {(\sigma_t^{\epsilon}*\mu(x))}^2 d\mu(x) \end{equation} and this quantity is bounded, once again, by Thm. \ref{kick}. This completes the proof of Lemma \ref{chainenergylemma}, up to the proof of Lemma \ref{fenergy}. 
\end{proof} 
\begin{lemma} \label{fenergy} Let $\mu$ be a compactly supported Borel measure such that $\mu(B(x,r)) \leq Cr^s$ for some $s \in (0,d)$. Suppose that $\alpha>d-s$. Then for $f \in L^2(\mu)$, 
\begin{equation} \label{hibob} \int {|\widehat{f\mu}(\xi)|}^2 {|\xi|}^{-\alpha} d\xi \leq C'{||f||}^2_{L^2(\mu)}. \end{equation}
\end{lemma} 
\begin{proof}
Let us observe that 
\begin{equation} \label{fenergysetup} \int {|\widehat{f\mu}(\xi)|}^2 {|\xi|}^{-\alpha} d\xi=C \int \int f(x)f(y) {|x-y|}^{-d+\alpha} 
d\mu(x)d\mu(y)= \langle Tf,f \rangle, \end{equation} where 
$$ Tf(x)=\int {|x-y|}^{-d+\alpha} f(y)d\mu(y)$$ and the inner product above is with respect to $L^2(\mu)$. Observe that 
$$ \int {|x-y|}^{-d+\alpha} d\mu(y) \approx \sum_{j>0} 2^{j(d-\alpha)} \int_{|x-y| \approx 2^{-j}} d\mu(y) \leq C \sum_{j>0} 
2^{j(d-\alpha-s)} \leq C'$$ since $\alpha>d-s$, where we have used $diam(supp(\mu))<\infty$. 

By symmetry, $\int {|x-y|}^{-d+\alpha} d\mu(x) \leq C'$ and Schur's test (\cite{Schur11}, see also \cite{St93}) implies at once that 
$$ {||Tf||}_{L^2(\mu)} \leq C' {||f||}_{L^2(\mu)},$$ which implies that conclusion of Lemma \ref{fenergy} in view of (\ref{fenergysetup}) and the Cauchy-Schwarz inequality. The proof of Lemma \ref{chainenergy} is thus complete. 
\end{proof} 
\vskip.125in 

We also need to show that the measure $d\mu^{\epsilon}_k$ is
non-trivial. A variant of this result is at the core of the proof of
the main result in \cite{BIT14}, as explained in the paragraph following Thm.~\ref{mainbit} above. See also \cite{MS99} where it was originally shown that the set of distances determined by a set of Hausdorff dimension $>\frac{d+1}{2}$ contains an interval. For the background on the Falconer distance problem and the latest results see \cite{Fal86}, \cite{Erd05} and \cite{W99}. 

\begin{lemma} \label{nontrivial} With the notation above, 
\begin{equation} \label{nontrivialest} \liminf_{\epsilon \to 0} \int d\mu^{\epsilon}_k(x)>0, \end{equation} provided that $\mu$ is a Frostman measure on a set of Hausdorff dimension $>\frac{d+1}{2}$. 
\end{lemma} 
\begin{proof} To prove the lemma, assume inductively that 
\begin{equation} \label{inductionlower} \liminf_{\epsilon \to 0} \int d\mu^{\epsilon}_{k-1}(x)>0. \end{equation} 

Note that this condition holds by definition if $k=1$ due to the fact that $\mu$ is a probability measure supported on $E$. By (\ref{inductionlower}) and Lemma \ref{chainenergylemma}, 
\begin{equation} \label{limitmeasure} \mu_{k-1} \equiv \lim_{ \epsilon \to 0} \mu^{\epsilon}_{k-1} \end{equation} is a non-zero Borel measure supported on $E$. This allows us to redefine $d\mu_k^{\epsilon}$ in (\ref{chainmeasures}) to equal 
$$\sigma_t^{\epsilon}*\mu_{k-1}(x) d\mu(x).$$

\vskip.125in 

We now write 
\begin{equation} \label{Msetup} \int d\mu^{\epsilon}_k(x)=\int \sigma_t^{\epsilon}*\mu_{k-1}(x) d\mu(x) \end{equation}
$$=\int \widehat{\sigma}_t(\xi) \widehat{\rho}(\epsilon \xi) \widehat{\mu}_{k-1}(\xi) \widehat{\mu}(\xi) d\xi$$ 
$$=\int \widehat{\sigma}(t\xi) \widehat{\mu}_{k-1}(\xi) \widehat{\mu}(\xi) d\xi+R^{\epsilon}(t)$$
$$=M(t)+R^{\epsilon}(t).$$ 

We now follow the argument in \cite{IMT12} to see that if $t>0$, $M(t)$ is continuous and $\lim_{\epsilon \to 0} R^{\epsilon}(t)=0$. 

\vskip.125in 

We have 
$$ M(t+h)-M(t)=\int (\widehat{\sigma}((t+h)\xi)-\widehat{\sigma}(t\xi)) \widehat{\mu}_{k-1}(\xi) \widehat{\mu}(\xi) d\xi.$$

The integrand goes to $0$ as $h \to 0$, so we proceed using the dominated convergence theorem. If $t>0$, the expression above is bounded by

$$ C(t) \int {|\xi|}^{-\frac{d-1}{2}} |\widehat{\mu}_{k-1}(\xi)| |\widehat{\mu}(\xi)| d\xi$$
$$ \leq C(t) {\left( \int {|\widehat{\mu}_{k-1}(\xi)|}^2 {|\xi|}^{-\frac{d-1}{2}} d\xi \right)}^{\frac{1}{2}} \cdot {\left( \int {|\widehat{\mu}(\xi)|}^2 {|\xi|}^{-\frac{d-1}{2}} d\xi \right)}^{\frac{1}{2}}$$ and this expression is finite by Lemma \ref{chainenergylemma}. We use the fact that,  if $t\ge t_0>0$, the estimate $|\widehat{\sigma}(t\xi)| \leq C {|\xi|}^{-\frac{d-1}{2}}$ holds with $C$ independent of $t$. This proves that $M(t)$ is continuous away from the origin. 

\vskip.125in

We now prove that $\lim_{\epsilon \to 0} R^{\epsilon}(t)=0$. We have 
$$ |R^{\epsilon}(t)| \leq \int \widehat{\sigma}(t \xi) |(1-\widehat{\rho}(\epsilon \xi))| |\widehat{\mu}_{k-1}(\xi)| 
|\widehat{\mu}(\xi)| d\xi$$
$$ \leq C \int_{|\xi|>\epsilon^{-1}}{|\xi|}^{-\frac{d-1}{2}} |\widehat{\mu}_{k-1}(\xi)| |\widehat{\mu}(\xi)| d\xi$$
$$ \leq C{\left( \int_{|\xi|>\epsilon^{-1}} {|\widehat{\mu}_{k-1}(\xi)|}^2 {|\xi|}^{-\frac{d-1}{2}} d\xi \right)}^{\frac{1}{2}} 
\cdot {\left( \int_{|\xi|>\epsilon^{-1}} {|\widehat{\mu}(\xi)|}^2 {|\xi|}^{-\frac{d-1}{2}} d\xi \right)}^{\frac{1}{2}}$$
$$ \leq C{\left( \int_{{\Bbb R}^d} {|\widehat{\mu}_{k-1}(\xi)|}^2 {|\xi|}^{-\frac{d-1}{2}} d\xi \right)}^{\frac{1}{2}} \cdot 
{\left( \int_{|\xi|>\epsilon^{-1}} {|\widehat{\mu}(\xi)|}^2 {|\xi|}^{-\frac{d-1}{2}} d\xi \right)}^{\frac{1}{2}}$$
$$ \leq C' {\left( \int_{|\xi|>\epsilon^{-1}} {|\widehat{\mu}(\xi)|}^2 {|\xi|}^{-\frac{d-1}{2}} d\xi \right)}^{\frac{1}{2}},$$ where in the last step we used Lemma \ref{chainenergylemma} once again. 

\vskip.125in 

We conclude that it is enough to show that 
\begin{equation} \label{putin} \lim_{\epsilon \to 0} \int_{|\xi|>\epsilon^{-1}} {|\xi|}^{-\frac{d-1}{2}} {|\widehat{\mu}(\xi)|}^2 d\xi=0. \end{equation}

Using  $s=dim_{{\mathcal H}}(E)-\delta$ for arbitrarily small $\delta>0$,  we have 
\begin{equation} \label{almostzero} \lim_{\epsilon \to 0} \sum_{j>\log_2(\epsilon^{-1})} \int_{2^j \leq |\xi| \leq 2^{j+1}} {|\xi|}^{-\frac{d-1}{2}} {|\widehat{\mu}(\xi)|}^2 d\xi. \end{equation}
Applying Lemma \ref{basicenergy}, to be proved below, we see that (\ref{almostzero}) is bounded by 
$$ \leq C \lim_{\epsilon \to 0} \sum_{j>\log_2(\epsilon^{-1})} 2^{-j \frac{d-1}{2}} \cdot 2^{j(d-s)}.$$ 
Hence, if $dim_{{\mathcal H}}(E)>\frac{d+1}{2}$, the limit is $0$. We have thus shown that 
$$ \lim_{\epsilon \to 0} \int d\mu^{\epsilon}_k(x)=M(t),$$ where $M(t) \ge 0$ is continuous function away from the origin. If we can show that $M(t)$ is not identically $0$, it will follow that there exists an open interval $I$, on which, $M(t)>c>0$. To see that $M(t)$ is not identically $0$, rewrite (\ref{Msetup}) in the form 
$$ \int \int \sigma^{\epsilon}_t(x-y) d\mu(x) d\mu_{k-1}(y).$$ 

This quantity is comparable to the Radon-Nikodym derivative of the measure on $\Delta(E)=\{|x-y|: x,y \in E\}$ given by 
$$ \liminf_{\epsilon \to 0} \epsilon^{-1} \mu \times \mu_{k-1} \{(x,y): t \leq |x-y| \leq t+\epsilon \}.$$ 

It follows that 
$$ \int M_k(t)dt=\int \int d\mu(x) d\mu_{k-1}(y)$$ and this quantity is strictly positive by (\ref{inductionlower}) and the fact that $\mu$ is a probability measure. This proves that $M_k(t)$ is not identically $0$ and thus completes the proof of Lemma \ref{nontrivial}. 


\end{proof}  

\begin{lemma} \label{basicenergy} Suppose that $\mu$ is a compactly supported Borel probability measure on ${\Bbb R}^d$ such that $\mu(B(x,r)) \leq Cr^s$. Then 
$$ \int_{|\xi| \leq R} {|\widehat{\mu}(\xi)|}^2 d\xi \leq CR^{d-s}.$$ 
\end{lemma} 
\begin{proof}
To prove the lemma, construct a smooth compactly supported function $h$ such that 
$$  \int_{|\xi| \leq R} {|\widehat{\mu}(\xi)|}^2 d\xi \leq \int
{|\widehat{\mu}(\xi)|}^2 \ \widehat{h}(\xi/R) d\xi.$$ This quantity is
bounded by $$ R^d \int \int h(R(x-y)) d\mu(x) d\mu(y) \leq CR^{d-s},$$
as claimed.  \end{proof} 
\subsection{Proof of Theorem \ref{necklace} (i)} 

Define 
\begin{equation} \label{necklacemeasure} {\mathcal N}_k^{\epsilon}(\mu)=\int \dots \int \left\{ \prod_{j=1}^{k+1} \sigma_t^{\epsilon}(x^{j+1}-x^j)d\mu(x^j) \right\}\sigma_t^{\epsilon}(x^{k+1}-x^1) d\mu(x^{k+1}). \end{equation} 

\vskip.125in 

Since $k$ is even, we may write $k=2n-2$ with $n$ an integer. Observe that 

\begin{equation} \label{csready} {\mathcal N}_k^{\epsilon}(\mu)=\int \int {\left\{\int \dots \int 
\prod_{j=1}^n \sigma_t^{\epsilon}(x^{j+1}-x^j) 
\prod_{j=2}^{n-1} d\mu(x^j) \right\}}^2 d\mu(x^1)d\mu(x^{n+1}). \end{equation}

\vskip.125in 

\subsubsection{Lower bound} 

\vskip.125in 

Applying Cauchy-Schwarz to (\ref{csready}) we see that it suffices to obtain a lower bound for 
\begin{equation} \label{cschain} \int \dots \int \left\{ \prod_{j=1}^n \sigma_t^{\epsilon}(x^{j+1}-x^j) d\mu(x^j) \right\} 
d\mu(x^{n+1}). \end{equation} 

In other words, the Cauchy-Schwarz inequality turns a chain into a necklace. The case $k=8$ is depicted in Fig. 3 above. 

\vskip.125in 

Observe that the quantity in (\ref{cschain}) equals $ \int d\mu^{\epsilon}_n(x)$ and we already proved 
in Lemma \ref{nontrivial} above that the $\liminf_{\epsilon \to 0}$ of this quantity is positive. 
\subsubsection{Upper bound} 
%
%
Define ${\mathcal N}_k^{\epsilon, \alpha}$ by the formula
\begin{equation} \label{3linesprep} {\mathcal N}_k^{\epsilon, \alpha}
  = \int \int F(x^1, x^n)^2 d\mu(x^1)d\mu(x^{n}), \end{equation} where 
\[ \begin{aligned} F(x^1, x^n) &= \int \dots \int 
\sigma^{\epsilon, -\alpha}(x^2-x^1) \sigma_t^{\epsilon, -\alpha}(x^n-x^{n-1}) 
\prod_{j=2}^{n-1} \sigma_t^{\epsilon, \alpha}(x^{j+1}-x^j) 
\prod_{j=2}^{n-1} d\mu(x^j), \text{ and }  \\
\sigma^{\alpha}(x) &=\frac{1}{\Gamma(\alpha)}
{(1-{|x|}^2)}_{+}^{\alpha-1}, \quad \sigma^{\epsilon, \alpha}=\sigma^{\alpha}*\rho_{\epsilon},
\end{aligned} \]
and $\alpha$ is a complex number. Recall the well-known fact (see e.g. \cite{St93,So93}) that 
\begin{equation} \label{alphadecay} |\widehat{\sigma}_t^{\alpha}(\xi)| \leq C{|\xi|}^{-\frac{d-1}{2}-Re(\alpha)}. \end{equation}

First consider the case $Re(\alpha)=1$, $n \ge 3$. Let $\alpha=1-iu$. Then 
\begin{align} |{\mathcal N}_k^{\alpha}(\mu)| &\leq \int \int {\left\{\int \dots \int 
G(x^1, x^2, x^{n-1}, x^{n})
\prod_{j=2}^{n-1} d\mu(x^j) \right\}}^2 d\mu(x^1)d\mu(x^{n}) \nonumber
\\ &\leq \int \int {\left\{ \int \int G(x^1, x^2, x^{n-1}, x^{n})
d\mu(x^2)d\mu(x^{n-1}) \right\}}^2 d\mu(x^1) d\mu(x^n) \label{1},
\text{ where } \\ G&=G(x^1, x^2, x^{n-1}, x^{n}) = |\sigma_t^{\epsilon,
  -1+iu}(x^2-x^1)| |\sigma_t^{\epsilon,
  -1+iu}(x^n-x^{n-1})|. \nonumber \end{align}
Observe that 
$$ |\sigma_t^{\alpha}(x)|=\frac{1}{|\Gamma(\alpha)|} {(1-{|x|}^2)}_{+}^{Re(\alpha)-1} $$ and we shall denote $|\sigma_t^{\epsilon, \alpha}(x)|=:\lambda^{\epsilon, \alpha}(x)$. 
In order to bound (\ref{1}), it suffices to show that 
$$ \int {(\lambda^{\epsilon,-1+iu}*\mu(x))}^2 d\mu(x) \leq C(u)$$ if $\mu$ is Frostman measure on a set of Hausdorff dimension 
$>\frac{d+3}{2}$. Since 
\begin{equation} \label{lambdadecay} |\widehat{\lambda^{\epsilon, -1+iu}}(\xi)| \leq C(u){|\xi|}^{-\frac{d-3}{2}} \end{equation} by (\ref{alphadecay}) and its proof, the claim follows from Theorem \ref{kick}. One can check using Stirling's formula that $C(u)$ grows like $Ce^{C|u|}$. 


We now consider the case $Re(\alpha)=-1$, $n \ge 3$. Then 
$$ |{\mathcal N}_k^{\alpha}(\mu)| \leq \int \int {\left\{\int \dots \int 
\prod_{j=2}^{n-1} |\sigma_t^{\epsilon, -1+iu}(x^{j+1}-x^j)|
\prod_{j=2}^{n-1} d\mu(x^j) \right\}}^2 d\mu(x^1)d\mu(x^{n})$$

\begin{equation} \label{eschonemnogo}= \int \int {\left\{\int \dots \int 
\prod_{j=2}^{n-1} \lambda^{\epsilon, -1+iu}(x^{j+1}-x^j)
\prod_{j=2}^{n-1} d\mu(x^j) \right\}}^2 d\mu(x^1)d\mu(x^{n}). \end{equation} 

\vskip.125in 

Let $g_1(x)=\lambda^{\epsilon, -1+iu}*\mu(x)$ and define inductively $g_j(x)=\lambda^{\epsilon, -1+iu}*(g_{j-1}\mu)(x)$. By inspection, the expression in (\ref{eschonemnogo}) equals 
\begin{equation} \label{eschochutchut} \int \int {|g_n(x^n)|}^2 d\mu(x^n) d\mu(x^1)=\int {|g_n(x^n)|}^2 d\mu(x^n). \end{equation}

Let $Tg(x)=\lambda^{\epsilon, -1+iu}*g(x)$. Then the right hand side of (\ref{eschochutchut}) equals 
$$ \int {|Tg_{n-1}(x)|}^2 d\mu(x).$$ 

Applying Thm.~\ref{kick} repeatedly, recalling (\ref{lambdadecay}) and that $\mu$ is a Frostman measure, we see that this expression is 
$ \leq C(n) {||g_1||}^2_{L^2(\mu)}$, provided that  
$$\hd(E)>d-\frac{d-3}{2}=\frac{d+3}{2}.$$
Applying Thm.~\ref{kick} one last time, we see that ${||g_1||}_{L^2(\mu)}$ is finite and the proof of the upper bound when $n \ge 3$ is completed by applying the following variant of the classical Hadamard three lines lemma due to Hirschman. 

\begin{lemma} \label{3lines} \cite{H52} If $\Phi$ is a continuous function on the strip $S$ that is holomorphic in the interior of $S$ and satisfies the bound
$$ \sup e^{-k|Im(z)| } \log|\Phi(z)|<\infty, z \in S$$ for some constant $k<\pi$, then
$$ \log|\Phi(\theta)| \leq \frac{\sin(\pi \theta)}{2} \int_{-\infty}^{\infty} \frac{\log|\Phi(iy)|}{\cosh(\pi y)-\cosh(\pi \theta)}+\frac{\log|\Phi(1+iy)|}{ \cosh(\pi y)+\cosh(\pi \theta)} dy$$ for all $\theta \in (0,1)$. 
\end{lemma} 


The proof of Thm. \ref{necklace} will be complete once we address the upper bound in the case $n=2$ and prove that at least some of the $k$-necklaces obtained are non-degenerate. 
\subsection{Proof of Theorem \ref{necklace} (i) for $n=2$} 
Consider 
\begin{equation} \label{setupsocialism} \int \int \left\{ \int \sigma_t^{\epsilon, \alpha}(x-z) \sigma_t^{\epsilon,-\alpha}(y-z) d\mu(z) \right\}^2 d\mu(x) d\mu(y). \end{equation} 

Suppose that $Re(\alpha)=1$. Then this quantity is bounded by 
$$  \int \int \left\{ \int |\sigma_t^{\epsilon,-\alpha}(y-z)| d\mu(z) \right\}^2 d\mu(x) d\mu(y)$$ 
\begin{equation} \label{almostsocialism} =\int \int \left\{ \int \frac{1}{|\Gamma(\alpha)|} {(1-{|y-z|}^2)}^{-Re(\alpha)-1}_{+} d\mu(z) \right\}^2 d\mu(x) d\mu(y). \end{equation}
This quantity is bounded above by the proof of the case $Re(\alpha)=1$, $n \ge 3$ above. Thus we are done by Lemma \ref{3lines} because we arrive at the exact same expression taking $Re(\alpha)=-1$ and reversing the roles of the variables. This takes care of the upper bound. The lower bound for a general $n$ is proved above. 
%
%
\subsection{Proof  of Theorem \ref{necklace} (ii)} 
Rewrite the expression in (\ref{almostsocialism}) above in the form 
$$ \int \int {(\lambda^{\epsilon, -1+iu}*\mu(y))}^2 d\mu(y) d\mu(x)=\int {(\lambda^{\epsilon, -1+iu}*\mu(x))}^2 d\mu(x).$$ 
Before we apply Thm. \ref{kick}, we need a simple calculation. Treating $K$ as a measure, observe that 
$$ \lambda^{\epsilon,-1+iu}(B(x,r)) \leq Cr^{d-2}.$$ 
The proof follows by a direct calculation. We now apply Thm. \ref{kick} with $K=\mu$, $\phi=\lambda^{\epsilon, -1+iu}$ and $\psi=\mu$. We shall assume that 
$$ |\widehat{\mu}(\xi)| \leq C{|\xi|}^{-\gamma}$$ for some $\gamma>0$. 
It follows that the $L^2(\phi) \to L^2(\psi)$ bound holds, with $f \equiv 1$ if 
$$ \gamma>d-\frac{d-2+s}{2}=\frac{d}{2}+1-\frac{s}{2}.$$ 
In particular, this means that if $s=d-\delta$, for some $\delta>0$, then 
$$ \gamma>1+\frac{\delta}{2}.$$ 


It remains to prove that at least some of the necklaces  obtained above are non-degenerate. 
%
%
\subsection{The non-degeneracy argument} 
Suppose, without loss of generality, that $|x^1-x^{j_0}| \leq N\epsilon$ for some $j_0 \not=1, k$ and that $|x^1-x^j|>N \epsilon$ for all $j<j_0$. See Fig.~4. 

Integrating in $d\mu(x^{j_0})$ and noting that $\sigma^{\epsilon}(x^{j_0}-x^{j_0+1}) \leq C\epsilon^{-1}$, $\sigma^{\epsilon}(x^{j_0}-x^{j_0-1}) \leq C\epsilon^{-1}$, we see that the expression in (\ref{necklacemeasure}), with the additional restriction that two vertices are within 
$N \epsilon$ of each other, is bounded by 
$$ C \cdot k \cdot {(N \epsilon)}^s \cdot \epsilon^{-2} \cdot C^{\epsilon}_{k-2}(\mu) \leq C'k N^s \epsilon^{s-2},$$ where $C_k^{\epsilon}(\mu)$ is defined in (\ref{chainradonnikodym}), and the fact that $C_{k-2}^{\epsilon}(\mu) \leq C$, independently of $\epsilon$ is proved in \cite{BIT14} and also follows easily from the fact, proved in the course of proving Lemma \ref{chainenergylemma} above that ${||f_k||}_{L^2(\mu)}$, with $f_k$ defined in (\ref{chainmeasures}) is bounded by a finite constant depending only on $k$. 

\begin{figure}
\label{chainfigure}
\centering
\includegraphics[scale=.5]{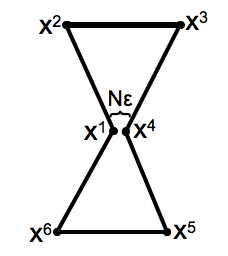}
\caption{A bottleneck}
\end{figure}
%
%
We conclude that the integral
\begin{align} \label{goodshit} \int_{S} &\left\{
    \prod_{j=1}^{k-1} \sigma_t^{\epsilon}(x^{j+1}-x^j)d\mu(x^j)
  \right\}\sigma_t^{\epsilon}(x^{k}-x^1) d\mu(x^{k}), \text{ where } \\ S &= \{(x^1, \dots, x^{k+1}) \in E^{k+1}: |x^1-x^j|> N\epsilon;
\ j\not=1 \}, \nonumber \end{align} is bounded from below by a non-zero constant as long as, say, $N<C\epsilon^{-1+\frac{2}{s}+\delta}$ for some $\delta>0$. If $\delta>0$ is chosen small enough, $\epsilon^{-1+\frac{2}{s}+\delta} \to \infty$ as $\epsilon \to 0$. Taking liminf as $\epsilon \to 0$ we see that there exists a non-degenerate $k$-necklace with gap $\equiv t$. 
\section{Concluding remarks} 

\vskip.125in 

The purpose of this section is to put the methods of this paper into perspective and describe their limitations. In simple terms the approach of this paper can be described as follows. We use the Cauchy-Schwarz inequality to relate a chain to necklace with even number of vertices. This procedure allows us to obtain an immediate lower bound on the Radon-Nikodym derivative of the natural candidate for the measure on set of necklaces with prescribed gaps. We then obtain an upper bound on the Radon-Nikodym derivative using the three lines lemma and harmonic analytic inequalities, thus completing the proof of the assertion that vertices of the necklace can be found inside a compact subset of ${\Bbb R}^d$ of a sufficiently large Hausdorff dimension. 

The method of proof described above suggests that further progress may be possible if we use the results of this paper and then create more elaborate point configuration by the means of the Cauchy-Schwarz or H\"older's inequalities. What types of configuration can we hope to obtain in this way? In order to get the flavor, let's start with a $4$-necklace and apply the Cauchy-Schwarz inequality in the $x^1,x^2,x^3$-variables. We obtain 
\begin{align} &{\left[ \int_{E^4} \prod_{j=1}^3 \sigma^{\epsilon}(x^j-x^{j+1})
      \sigma^{\epsilon}(x^4-x^1) \prod_{j=1}^{4}d\mu(x^j)
    \right]}^2 \nonumber \\
\label{setupcommunism} \qquad \qquad &\leq  \int_E \left\{ \int_{E^3} \prod_{j=1}^3 \sigma^{\epsilon}(x^j-x^{j+1})
      \sigma^{\epsilon}(x^4-x^1) \prod_{j=1}^{3}d\mu(x^j) \right\}^2 d\mu(x^4) \\ 
\label{twonecklaces} \qquad \qquad &\begin{aligned}
  \leq \int_{E^7} &\sigma^{\epsilon}(x^1-x^4) \cdot \prod_{j=1}^3
  \sigma^{\epsilon}(x^{j+1}-x^j) \times \\  &\qquad \qquad \qquad \sigma^{\epsilon}(x^4-x^7)
  \cdot \prod_{j=4}^6 \sigma^{\epsilon}(x^{j+1}-x^j)
  \prod_{j=1}^{7}d\mu(x^j), \end{aligned} \end{align}
%
%
which is the Radon-Nikodym derivative of the natural measure on two $4$-necklaces sharing the vertex $x^4$. See Figure 5. 
\begin{figure}
\label{twonecklaces}
\centering
\includegraphics[scale=.5]{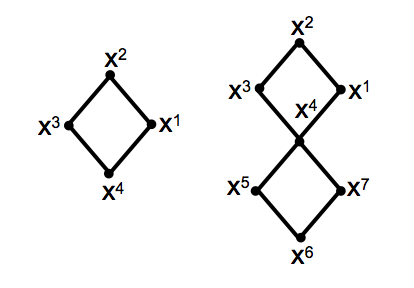}
\caption{Cauchy-Schwarz turns a necklace into two necklaces sharing a point}
\end{figure}
Obtaining an upper bound for (\ref{twonecklaces}) is by no means trivial, but possible. We outline the argument because it leads to interesting harmonic analysis and illustrates the rich set of connections between geometric problems and harmonic analytic inequalities that these questions foster. Recalling the idea behind (\ref{setupsocialism}), we can express (\ref{setupcommunism}) in the form 
\begin{equation} \int \int \int F^2(x^2,x^4) G^2(x^7,x^4) d\mu(x^2)d\mu(x^4)d\mu(x^6),\end{equation}
where 
\begin{align*} F(x^2,x^4) &=\int \sigma^{\epsilon}(x^2-x^1)
  \sigma^{\epsilon}(x^4-x^1) d\mu(x^1),  \text{ and } \\  
G(x^7,x^4) &=\int \sigma^{\epsilon}(x^7-x^4) \sigma^{\epsilon}(x^7-x^6) d\mu(x^7). 
\end{align*} 
Applying Cauchy-Schwarz yet again reduces matters to bounding the quantity 
\begin{equation} \label{powerfour} \int \int \left\{ \int \sigma^{\epsilon}(x^2-x^1) \sigma^{\epsilon}(x^4-x^1) d\mu(x^1) \right\}^4 
d\mu(x^2) d\mu(x^4). \end{equation}

We pause for a moment to point out the difference between this quantity and (\ref{setupsocialism}), the expression we needed to bound to handle the rhombus ($4$-necklace). In (\ref{setupsocialism}) the inner expression in (\ref{powerfour}) is raised to the power of $2$ instead of the power of $4$. This naturally leads us to consider the $L^4$ version of Thm. \ref{kick}, which can be obtained, with a worse yet still non-trivial lower bound on exponents $s_{\phi}$ and $s_{\psi}$, corresponding to the dimensional restriction, by a rather straightforward modification of the proof. 

By the same method we can start with any necklace with an even number of vertices and by applying H\"older's inequality with the integer exponent $m \ge 2$ (positive integer), we obtain $m$ necklaces sharing a common vertex. While it would be difficult to classify succinctly all the point configurations that can be obtained by starting with a chain and successively applying H\"older's inequality, this example is quite representative and also illustrates the limitations of our method. 

There remain geometric configurations that cannot be handled either by
the methods of this paper, or those in \cite{CLP14}. For example, the
three-dimensional corner, described in (\ref{3dcorner}) appears to be
outside the reach of both methods. The authors hope to return to this issue in a sequel.


\vskip.25in 

\begin{thebibliography}{8}

\bibitem{BIT14} M. Bennett, A. Iosevich and K. Taylor, {\it Finite point configurations inside thin subsets of Euclidean space}, (in preparation), (2014). 

\bibitem{B86} J. Bourgain, {\it A Szemeredi type theorem for sets of positive density}, Israel J. Math. \textbf{54} (1986), no. 3, 307-331.

\bibitem{CLP14} V. Chan, I. \L aba and M. Pramanik, {\it Finite
    configurations in sparse sets}, to appear in J. d'Analyse Math,
  online version available at http://arxiv.org/pdf/1307.1174.pdf. 

\bibitem{Erd05} B. Erdo\~{g}an {\it A bilinear Fourier extension
    theorem and applications to the distance set problem},
  Int. Math. Res. Not. (2006).

\bibitem{Fal86} K. J. Falconer, {\it On the Hausdorff dimensions of distance sets} Mathematika \textbf{32} (1986) 206-212.


\bibitem{FKW90} H. Furstenberg, Y. Katznelson, and B. Weiss, {\it Ergodic theory and configurations in sets of positive density} Mathematics of Ramsey theory, 184-198, Algorithms Combin. {\bf 5}, Springer, Berlin, (1990).


\bibitem{H52} I. I. Hirschman, Jr., {\it A convexity theorem for certain groups of transformations}, Journal d'Analyse \textbf{2} (1952) 209-218.

\bibitem{IMT12} A. Iosevich, M. Mourgoglou and K. Taylor, {\it On the Mattila-Sj\"{o}lin theorem for distance sets}, Ann. Acad. Sci. Fenn. Math. \textbf{37}, no.2 , (2012). 




\bibitem{ISTU14} A. Iosevich, E. Sawyer, K. Taylor and I. Uriarte-Tuero, {\it Borel measures of polynomial growth and classical convolution inequalities}, (in preparation), (2014). 

\bibitem{Mag10} P. Maga {\it Full dimensional sets without given patterns}, Real Anal. Exchange, \textbf{36}, 79�90, (2010). 


\bibitem{M95} P. Mattila, {\it Geometry of sets and measures in Euclidean spaces}, Cambridge University Press,  (1995). 


\bibitem{MS99} P. Mattila and P. Sj\"{o}lin, {\it Regularity of distance measures and sets},
Math. Nachr. \textbf{204} (1999), 157-162.


\bibitem{Schur11} I. Schur, {\it Bemerkungen zur Theorie der Beschr�nkten Bilinearformen mit unendlich vielen Ver�nderlichen}, J. reine angew. Math. \textbf{140} (1911), 1-28. 

\bibitem{So93} C. Sogge, {\it Fourier integrals in classical analysis}, Cambridge University Press, (1993). 

\bibitem{St93} E. M. Stein, {\it Harmonic Analysis}, Princeton University Press, (1993). 

\bibitem{Str90} R. Strichartz, {\it Fourier asymptotics of fractal measures}, Journal of Func. Anal. \textbf{89}, (1990), 154-187. 


\bibitem{W99} T. Wolff, {\it Decay of circular means of Fourier transforms of measures}, Int. Math. Res. Not.  \textbf{10} (1999) 547--567.

\bibitem{Z06} T. Ziegler, {\it Nilfactors of ${\Bbb R}^d$ actions and configurations in sets of positive upper density in ${\Bbb R}^m$}, J. Anal. Math. \textbf{99}, 249-266 (2006). 


\end{thebibliography}
\end{document}